\documentclass[a4paper,10pt]{article}

\usepackage{indentfirst,amsmath,amsfonts,amstext,amssymb,amscd,bezier,amsthm}
\usepackage[latin1]{inputenc}
\usepackage[dvips]{graphicx}
\usepackage[latin1]{inputenc}
\usepackage{setspace}
\usepackage{graphicx}
\usepackage{tabularx}
\usepackage{longtable}
\usepackage[usenames,dvipsnames]{pstricks}
\usepackage{epsfig}
\usepackage{pst-grad}
\usepackage{pst-plot}
\usepackage[all]{xy}
\usepackage{hyperref}
\usepackage[T1]{fontenc}
\usepackage{ae,aecompl}
\usepackage{indentfirst}
\usepackage{xspace}
\usepackage{color}
\usepackage{psfrag}
\usepackage{mathrsfs}
\usepackage{upgreek}                            
\usepackage{makeidx}                            
\usepackage{fancyhdr}
\usepackage{fancybox}
\usepackage[rm]{titlesec}
\usepackage{float}   
\usepackage{enumerate}     

\input xy
\xyoption{all}

\newcommand{\p}{\mathbb{P}}

\newcommand{\F}{\mathbb{F}}

\newcommand{\lra}{\longrightarrow}

\newcommand{\ff}{\mathcal{F}}

\newcommand{\fq}{\mathbb{F}_q}

\newcommand{\kk}{\mathbb K}
\newcommand{\xx}{\mathcal X}
\newcommand{\yy}{\mathcal Y}

\newcommand{\cc}{\mathcal C}

\newcommand{\aut}{\operatorname{Aut}}

\newcommand{\fqc}{\overline{\mathbb{F}}_q}

\theoremstyle{plain}
\newtheorem{thm}{Theorem}[section]

\newtheorem{prop}[thm]{Proposition}

\newtheorem{rem}[thm]{Remark}
\newtheorem{ex}[thm]{Example}

 \font\numberfont= pzcmi scaled
3000
\titleformat{\chapter}[display]
  {\normalfont\Large 
  }
  {
   \filright
   \rule[32pt]{.7\linewidth}{4pt}
   \hspace{-8pt}
   \shadowbox{
   \begin{minipage}{.15\linewidth}
     \begin{center}
          \textsl{\bf {\large \chaptertitlename}}\\
       \vspace{1ex}
       {\bf {\numberfont \thechapter}}\\
       \vspace{1ex}
     \end{center}
   \end{minipage}}
  }
  {-10pt}
  {\filcenter
           \sl
           \bf
              \Huge
     }
  [\vspace{-1cm}\hfill\rule{.8\textwidth}{0.5pt}\\
\vskip-2.8ex\hfill\rule{.7\textwidth}{4pt}\vspace*{-1ex}]
\titlespacing{\chapter}{0pt}{*4}{*1}

\titleformat{\section}[block]
{\normalfont\bfseries} {\thesection}{0.5em}{}

\titleformat{\subsection}[block]
{\normalfont\large\bfseries} {\thesubsection}{0.5em}{}

\makeindex                                      
\makeatletter
\numberwithin{equation}{section}

\begin{document}

\title{Separable degree of the Gauss map and strict dual curves over finite fields}

\author{\textbf{Nazar Arakelian} \\
 \small{Centro de Matem\'atica, Computa\c c\~ao e Cogni\c c\~ao - Universidade Federal do ABC}\\
\small{Av. dos Estados 5001, CEP 09210-580, Santo Andr\'e-SP, Brazil}\\
\small{n.arakelian@ufabc.edu.br} 
}

\maketitle
 \begin{abstract}
 Let $\xx$ be a projective algebraic curve and denote by $\xx^{'}$  its strict dual curve. The map $\gamma:\xx \lra \xx^{'}$ is called (strict) Gauss map of $\xx$. In this manuscript, we study the separable degree of the Gauss map of curves defined over finite fields. In particular, we give a generalization of a known result on the separable degree of the Gauss map of plane Frobenius nonclassical curves. We also obtain a characterization of certain plane strange curves. 
 \end{abstract}


\section {Introduction}\label{intro}

For an algebraically closed field $\kk$, let $\xx \subset \p^n(\kk)$ be a (geometrically irreducible algebraic) curve defined over $\kk$, where $\p^n(\kk)$ denotes the $n$-dimensional projective space defined over $\kk$. If $(\p^{n}(\kk))^{'}$ denotes the dual projective space, then the (strict) dual of $\xx$, denoted by $\xx^{'}$, is defined as the closure in $(\p^{n}(\kk))^{'}$ of the set of osculating hyperplanes at the nonsingular locus of $\xx$. The dual curve $\xx^{'}$ is a subcover of $\xx$, and the cover $\gamma:\xx \lra \xx^{'}$ is called the (strict) Gauss map of $\xx$. 

If the characteristic of $\kk$ is $0$, then it is well known that $\gamma$ is birational. Moreover, if $\gamma^{'}:\xx^{'} \lra \xx^{''}$ is the Gauss map of the dual curve $\xx^{'}$, then $\gamma^{'} \circ \gamma$ is the identity map, i.e., $\xx=\xx^{''}$, see \cite{Be} or \cite[Remark 2]{GV0}. When $\kk$ has characteristic $p>0$, the situation is the same as long as $p>n$ and $\xx$ is a classical curve, that is, the intersection multiplicity of $\xx$ and the osculating hyperplane at a generic point is $n$. In contrast, when $\xx$ is nonclassical, then $\gamma$ is not necessarily separable and its separable degree can be positive. Indeed, in \cite{GV0} several examples endorsing this facts are given. Furthermore, it follows as a particular case of a result of Kaji in \cite{Ka} that, given a plane curve $\cc$ and an arbitrary inseparable finite field extension $\mathbb{L}/\kk(\cc)$, there exists a plane curve $\xx$ such that $\gamma(\xx)=\cc$, $\kk(\xx) \cong \mathbb{L}$ and the extension $\kk(\xx)/\kk(\cc)$ coincides with $\mathbb{L}/\kk(\cc)$, where $\kk(\cc)$ (resp. $\kk(\xx)$) denotes the function field of $\cc$ (resp. $\xx$) (see also \cite{FI} for a more general result).

The properties of the Gauss map $\gamma$ and the dual $\xx^{'}$ of a curve $\xx$ over a field $\kk$ of characteristic $p>0$ have been investigated by many authors in the last decades, sometimes either in a more general or a parallel context, see \cite{FI,He,HV,Ka}. For instance, given a curve $\xx \subset \p^{n}(\kk)$ the Gauss map $\iota^{(m)}:\xx \lra \mathbb{G}(\p^n(\kk),m)$ of order $m$ ($1 \leq m \leq n-1$), which associates nonsingular points of $\xx$ with the respective $m$-osculating spaces, where $\mathbb{G}(\p^n(\kk),m)$ is the Grassmannian variety of $m$ spaces in $\p^n(\kk)$, is investigated by Homma and Kaji in \cite{HK} (note that $\iota^{(n-1)}=\gamma$). More precisely, if $(\varepsilon_0,\ldots,\varepsilon_n)$ denotes the order-sequence of $\xx$, i.e., the sequence of all the possible intersection multiplicities of $\xx$ with a hyperplane in $\p^n(\kk)$ at a generic point, \cite[Theorem 1]{HK} states that the inseparable degree of $\iota^{(m)}$ is the highest power of $p$ dividing $\varepsilon_{m+1}$. In particular, if $\varepsilon_n=kp^r$, where $k$ is prime to $p$, then the inseparable degree of the Gauss map $\gamma$ is $p^r$. 

In several situations, the computation of the separable degree of the Gauss map is desirable, specially when $\kk$ is the algebraic closure of a finite field. For example, in the proof of the Natural Embedding Theorem for maximal curves, in  \cite[Section 3]{KT}. Also in \cite[Section 3.3]{HKo} to obtain bounds on the size of complete arcs, among others. In this direction, the following Theorem is a consequence of a result by Hefez and Voloch on Frobenius nonclassical curves \cite[Proposition 3]{HV}. They actually were investigating Gaussian duals of such curves, to say, the closure of the image of the nonsingular locus of $\xx$ by the Gauss map in the Grassmannian of lines in $\p^n(\kk)$. Hence, for $n=2$, the strict dual and the Gaussian dual are the same.

\begin{thm}[Hefez-Voloch]\label{hefezvoloch}
Let $\xx$ be a plane curve defined over a finite field $\fq$. If the image of infinitely many points of $P \in \xx$ by the $\fq$-Frobenius map belongs to the tangent line to $\xx$ at $P$ (that is, $\xx$ is $\fq$-Frobenius nonclassical), then the Gauss map of $\xx$ is purely inseparable.
\end{thm}

In \cite{GV0}, Garcia and Voloch presented several relations between the osculating spaces of a curve $\xx$ and the ones of its dual $\xx^{'}$. As a consequence, they establish some criteria for when the equality $\xx=\xx^{''}$ holds, and so when $\gamma$ has separable degree equals to $1$, provided that $\xx^{'}$ is non-degenerate. Furthermore, they give a different proof of Theorem \ref{hefezvoloch}, and show that the converse holds for some extension of $\fq$, with the assumption that $\xx$ is smooth and nonclassical \cite[Theorem 4]{GV0}.  

Related to this subject, we have the so called strange curves. A curve $\xx$ is called strange if all tangent lines to nonsingular points of $\xx$ are concurrent. In this case, the point $P$ in which all generic tangent to $\xx$ meet is called nucleus of $\xx$. Note that if $\xx$ is a plane curve, this is equivalent to the dual curve $\xx^{'}$ be a line. The properties of strange curves have been broadly investigated over the last years. For instance, in \cite{BH} many properties of such curves are established, such as general shape of their equation and their genera. 

In this paper, we basically study some maps on $\xx$ that induce maps on $\xx^{'}$, namely the Frobenius map and linear automorphisms. As a byproduct, we will provide a generalization of Theorem \ref{hefezvoloch} for certain Frobenius nonclassical curves defined over $\fq$. It also will be shown that for a curve defined over $\fqc$, if $\fqc(\xx^{'})$ is a subfield of a Galois subcover of $\fqc(\xx)$ by a linear automorphism group, then $\xx^{'}$ is degenerate, and some consequences of this fact will be exploited. This results will be obtained in Sections 3 and 4, after we establish some preliminary results and notation in Section 2. 

\section{Background and notation}\label{back}

For a prime number $p$, let $\fq$ be the finite field of order $q$, where $q=p^h$ with $h>0$. Let $\xx \subset \p^n(\fqc)$ be a (geometrically irreducible algebraic) curve defined over $\fq$, where $\fqc$ denotes the algebraic closure of $\fq$. Given an algebraic extension $\mathbb{H}$ of $\fq$, the function field of $\xx$ over $\mathbb{H}$ will be denoted by $\mathbb{H}(\xx)$. For a nonsingular point $P \in \xx$, there exists a sequence $(j_0(P), \ldots,j_n(P))$ of integers with $0 \leq j_0(P)<\cdots<j_n(P)$, called order sequence of $P$, which is defined by all the possible intersection multiplicities of $\xx$ and some hyperplane of $\p^n(\fqc)$ at $P$. When there is no risk of confusion, $j_i$ denotes $j_i(P)$, for $i=0,\ldots,n$. It turns out that, except for finitely many points on $\xx$, this sequence is the same, \cite[Section 1]{SV}. It is called order sequence of $\xx$, and denoted by $(\varepsilon_0,\ldots, \varepsilon_n)$. If $(\varepsilon_0,\ldots, \varepsilon_n)=(0,\ldots,n)$, the curve $\xx$ is said to be a classical curve. Otherwise, $\xx$ is nonclassical. Given a nonsingular point $P \in \xx$, there exists an unique hyperplane $H_P(\xx) \subset \p^n(\fqc)$ such that the intersection multiplicity of $H_P(\xx)$ and $\xx$ at $P$ is $j_n$. Such hyperplane is denoted by osculating hyperplane to $\xx$ at $P$. 

Let $t \in \fqc(\xx)$ be a separating element and $k$ a non-negative integer. Given $f \in \fqc(\xx)$, the $k$-th Hasse derivative of $f$ with respect to $t$ is denoted by $D_t^{(k)}f$. The following result is a consequence of \cite[Theorem 1.1]{SV}.

\begin{prop}\label{oschypeq}
Suppose that $\xx$ is defined by the coordinate functions $x_0,\ldots,x_n \in \fq(\xx)$. Let $P \in \xx$ be a nonsingular point and let $t$ be a local parameter at $P$. The osculating hyperplane at $P$ is spanned by the points 
$
\big((D_t^{(j_i)}x_0)(P):\cdots: (D_t^{(j_i)}x_n)(P) \big),
$
 for $i=0,\ldots,n-1$.
\end{prop}

Since $\xx$ is defined over $\fq$, i.e., $\xx$ is defined by coordinate functions $x_0,\ldots,x_n \in \fq(\xx)$, by \cite[Proposition 2.1]{SV}, there exist a sequence of non-negative integers $(\nu_0,\ldots,\nu_{n-1})$, chosen minimally in the lexicographic order, such that
\begin{equation}\label{fncl}
\left|
  \begin{array}{ccc}
  x_0^q & \ldots & x_n^q \\
  D_t^{(\nu_0)}x_0 & \ldots & D_t^{(\nu_0)}x_n \\
   \vdots & \cdots & \vdots \\
  D_t^{(\nu_{n-1})}x_0 & \cdots & D_t^{(\nu_{n-1})}x_n
  \end{array}
  \right| \neq 0.
  \end{equation}
This sequence is called $\fq$-Frobenius order sequence of $\xx$. It does not depend on $t$ and is invariant under change of projective coordinates of $\xx$ (see \cite[Proposition 1.4]{SV}). The curve $\xx$ is called $\fq$-Frobenius classical if $(\nu_0,\ldots,\nu_{n-1})=(0,\ldots,n-1)$, and $\fq$-Frobenius nonclassical otherwise. From \cite[Proposition 2.1]{SV}, there exists $I \in \{1, \ldots,n\}$ for which $\{\nu_0,\ldots,\nu_{n-1}\}=\{\varepsilon_0,\ldots, \varepsilon_n\} \backslash \{\varepsilon_I\}$.

Let $(\p^n(\fqc))^{'}$ be the dual projective space and consider the open set $$\xx^{\circ}=\{P \in \xx \ | \ P \text{ is nonsingular and } j_i(P)=\varepsilon_i , \ i=0,\ldots,n\}.$$ In our context, by a generic point of $\xx$ we mean a point $P \in \xx^{\circ}$. The strict Gauss map $\gamma:\xx^{\circ} \lra (\p^n(\fqc))^{'}$ is the map defined at a generic point $P \in \xx$ by $\gamma(P)=(H_P(\xx))^{'}$, where $L^{'}$ denotes the dual of a hyperplane $L$. The closure of the image of $\gamma$ is the strict dual of $\xx$, denoted by $\xx^{'}$. For simplicity, we will denote the map $\xx \lra \xx^{'}$ also by $\gamma$. Note that $\gamma$ is a morphism defined in the nonsingular locus of $\xx$. Here, we will abuse and call $\gamma$ a morphism of $\xx$. The separable (respectively inseparable) degree of $\gamma$ will be denoted by $\deg_s \gamma$ (respectively $\deg_i \gamma$). The following facts concerning the map $\gamma$ (some of them pointed out in the introduction) hold:
\begin{itemize}
\item $\deg_i \gamma$ is the highest power of $p$ dividing $\varepsilon_n$ (\cite[Theorem 1]{HK}).
\item If $\xx$ is classical and  $p>n$, then $\deg_i \gamma=\deg_s \gamma =1$, and $\xx=\xx^{''}$ (\cite{GV0}).
\item If $n=2$ and $\xx$ is $\F_{q^r}$-Frobenius nonclassical for some $r>0$, then $\varepsilon_2$ is a power of $p$, $\deg_s \gamma=1$ and $\deg_i \gamma = \varepsilon_2$ (\cite[Propositions 1, 3 and 4]{HV}).
\item If $d=\deg \xx$, then
\begin{equation}\label{trivialbound}
\deg_s \gamma \leq \frac{d}{\varepsilon_n}
\end{equation} Indeed, by definition, $\deg_s \gamma$ is the number of genereic points of $\xx$ sharing the same osculating hyperplane. Thus, \eqref{trivialbound} follows from B\'ezout's Theorem.
\end{itemize}

For a given $r>0$, the $\F_{p^r}$-Frobenius map $\Phi_{p^r}:\xx \lra \p^n(\fqc)$ is defined by $\Phi_{p^r}:P \mapsto P^{p^r}$. It induces a map from the function field $\fqc(\xx)$ onto $\fqc(\xx)^{p^r}=\{f^{p^r} \ | \ f \in \fqc(\xx)\}$. Therefore, $\Phi_{p^r}$ is purely inseparable of degree $p^r$. We point out that although $\Phi_q$ is purely inseparable, we have $\Phi_q: \xx \lra \xx$, as $\xx$ is defined over $\fq$.

Suppose that $p>n$ and $\varepsilon_n=p^r$ for some $r>0$. If $\xx$ is defined by the coordinate functions $x_0,\ldots,x_n$, then \cite[Theorem 7.65]{HKT} ensures that there exist $z_0,\ldots,z_n \in \fq(\xx)$ such that $z_i$ is a separating element for at least one $i$ and
\begin{equation}\label{e}
z_0^{p^r}x_0+\cdots+z_n^{p^r}x_n=0.
\end{equation}
Furthermore, for a generic point $P \in \xx$, the osculating hyperplane to $\xx$ at $P$ is defined by
$$
H_P(\xx):(z_0^{p^r}(P))X_0+\cdots+(z_n^{p^r}(P))X_n=0.
$$
Therefore, in such case we conclude that $\gamma=\Phi_{p^r} \circ \gamma_s=(z_0^{p^r}:\cdots:z_n^{p^r})$, where $\gamma_s=(z_0:\cdots:z_n)$ is a separable morphism. In particular, $\deg_s \gamma= \deg \gamma_s$ and $\deg_i \gamma =\deg \Phi_{p^r}=p^r$.

\section{Generalization of Theorem \ref{hefezvoloch} }

Let $\xx \subset \p^n(\fqc)$ a curve defined over $\fq$. Let $(\varepsilon_1,\ldots,\varepsilon_n)$ and $(\nu_0,\ldots,\nu_{n-1})$ denote the order sequence and $\fq$-Frobenius order sequence of $\xx$, respectively. Assume that $\nu_{n-1}= \varepsilon_n$. Then $(\nu_0,\ldots,\nu_{n-1})\neq (\varepsilon_0,\ldots,\varepsilon_{n-1})$, and so $\xx$ is $\fq$-Frobenius nonclassical. Moreover, from Proposition \ref{oschypeq} and \eqref{fncl}, the hypothesis $\nu_{n-1}= \varepsilon_n$ is equivalent to the following: $\Phi_q(P) \in H_P(\xx)$ for a generic point $P \in \xx$. Under such condition, we obtain a slight improvement of bound \eqref{trivialbound}.

\begin{prop}
Let $\xx \subset \p^n(\fqc)$ be a curve defined over $\fq$ with order sequence $(\varepsilon_0,\ldots,\varepsilon_n)$. Let $1 \leq \ell < n$ and suppose that $\xx$ is $\F_{q^i}$-Frobenius nonclassical for $i=1,\ldots,\ell$ such that the last term of the $\F_{q^i}$-Forbenius order sequence equals to $\varepsilon_n$. Then
\begin{equation}\label{impr}
\deg_s \gamma \leq \frac{d}{\varepsilon_n+\ell}.
\end{equation} 
\end{prop}
\begin{proof}
Given a generic point $P \in \xx$ with osculating hyperplane $H_P(\xx)$, we have that for a fixed $i$, the osculating hyperplane at $\Phi_{q^i}(P)$ is $\Phi_{q^i}(H_P(\xx))$. Note that, since $P$ is generic, its osculating hyperplane is not $\F_{q^i}$-rational, and then $\Phi_{q^i}(H_P(\xx)) \neq H_P(\xx)$. Hence, to each generic point sharing the same osculating hyperplane $H_P(\xx)$, we have an extra point in $H_P(\xx) \cap \xx$ for which $H_P(\xx)$ is not the osculating hyperplane, namely $\Phi_{q^i}(P)$. Since $\Phi_{q^i}$ is injective, all these points are distinct. Finally, if $Q \in \xx$ is such that $H_P(\xx)$ is osculating to $\xx$ at $Q$, then $\Phi_{q^i}(P)\neq \Phi_{q^j}(Q)$ if $i \neq j$, with $1 \leq i<j \leq \ell$. Indeed, if $\Phi_{q^i}(P)= \Phi_{q^j}(Q)$, then $\Phi_{q^{j-i}}(Q)= P$, and so $H_P(\xx)$ is fixed by $\Phi_{q^{j-i}}$, a contradiction as $P \in \xx$ is generic. Therefore \eqref{impr} follows from the B\'ezout's Theorem.
\end{proof}

However, note that the result in Theorem \ref{hefezvoloch} does not depend on the degree of the curve. In fact, it only depends on the dimension of the ambient projective space. In what follows, $\lceil m \rceil$ denotes the smallest integer bigger or equal to the rational number $m$.

\begin{thm}\label{main}
Let $\xx \subset \p^n(\fqc)$ be a curve defined over $\fq$ with order sequence $(\varepsilon_0,\ldots,\varepsilon_n)$. Let $1 \leq \ell < n$ and suppose that $\xx$ is $\F_{q^i}$-Frobenius nonclassical for $i=1,\ldots,\ell$ such that the last term of the $\F_{q^i}$-Forbenius order sequence equals to $\varepsilon_n$. If
$$
\deg_s \gamma > \left\lceil \frac{n}{\ell}-1 \right\rceil,
$$
then for a generic $P \in \xx$, the set $S(P)=\{\Phi_{q^k}(Q) \ | \ Q \in \gamma^{-1}(\gamma(P)), \  k=0,\ldots,\ell-1\}$ lies in a linear subspace of $\p^n(\fqc)$ of dimension $n-2$.
\end{thm}
\begin{proof}
Suppose that $\deg_s \gamma >  \left\lceil \frac{n}{\ell}-1 \right\rceil$. Let $P \in \xx$ be a generic point and assume that $S(P)$ is not in a linear subspace of $\p^n(\fqc)$ of dimension $n-2$. The hypothesis on the $\F_{q^i}$-Frobenius order sequence gives that $\Phi_{q^i}(P) \in H_P(\xx)$ for all $i=1,\ldots ,\ell$. Since $\deg_s \gamma  > \left\lceil \frac{n}{\ell}-1 \right\rceil$, there are at least $\lceil n/\ell \rceil$ points sharing the same osculating hyperplane $H_P(\xx)$. Then, as $S(P) \subset H_P(\xx)$ and it is not in a linear subspace of dimension $n-2$, we have for sure $n$ points of $S(P)$ in general position in $H_P(\xx)$, and $H_P(\xx)$ is completely determined by these points.  Hence $\Phi_q(H_P(\xx))$ is completely determined by the set $\{\Phi_q(R) \ | \ R \in S(P)\}$. But if $R \in S(P)$, we have $R=\Phi_{q^k}(Q)$ for some $Q \in \gamma^{-1}(\gamma(P))$ and $k=0,\ldots,\ell-1$, thus $\Phi_q(R)=\Phi_{q^{k+1}}(Q) \in H_P(\xx)$. Therefore, $\Phi_q(H_P(\xx))=H_P(\xx)$. Since $P$ is generic, this implies that $\Phi_q$ fixes infinitely many points in $\xx^{'}$, that is, $\xx^{'}$ has infinitely many $\fq$-rational points, which is a contradiction. 
\end{proof}

In particular, for $\ell=1$ we have the following.

\begin{thm}\label{gen1}
Let $\xx \subset \p^n(\fqc)$ be an $\fq$-Frobenius nonclassical curve defined over $\fq$ such that $\nu_{n-1}=\varepsilon_n$. If $\deg_s \gamma \geq n$ then the generic pre-image of $\gamma$ lies in a linear subspace of $\p^n(\fqc)$ of dimension $n-2$.
\end{thm}

\begin{rem}\label{obsgen}
If $n=2$, then the hypothesis $\nu_{n-1}=\varepsilon_n$ is equivalent to $\xx$ being $\fq$-Frobenius nonclassical. Moreover, a linear subspace in $\p^2(\fqc)$ of dimension $n-2$ is just a point. Hence, we conclude from Theorem \ref{gen1} that $\deg_s \gamma=1$. In other words, we obtain the result stated in Theorem \ref{hefezvoloch}.
\end{rem}

To illustrate Theorem \ref{gen1}, we now show how to construct examples of curves in $\p^3(\fqc)$ with $\deg_s \gamma = k$, with $k$ prime to $p$, such that the points in a generic pre-image of $\gamma$ are collinear. Let $\ff:z_0^{p^r}x+z_1^{p^r}y+z_2^{p^r}=0$ be a plane $\fq$-Frobenius nonclassical curve with order sequence $(0,1,p^r)$, where $p^r \leq q$, $z_0,z_1,z_2 \in \fq(\ff)=\fq(x,y)$ and $x$ is a separating variable. Then
$$
D_x^{(1)}(y)=\frac{y-y^q}{x-x^q}=-\frac{z_0^{p^r}}{z_1^{p^r}},
$$ 
and $D_x^{(l)}(y)=0$ for $1<l<p^r$. For $k>1$ prime to $p$, let $f \in \fqc(x,y)$ such that $f \neq w^d$ for all $w \in \fq(x,y)$ and $d>1$ divisor of $k$ (for instance, we may chose $f=1/t$, where $t$ is a local parameter at some nonsingular point $P \in \ff$). Suppose that $z$ is a root of $T^k-f \in \fqc(x,y)[T]$ such that $D_x^{(2)}(z) \neq  0$ (note that $x$ is also separating in $\fqc(x,y,z)$). In particular, 
$$
D^{(1)}_x(z) \neq \frac{z-z^q}{x-x^q}.
$$
Consider $\xx=\phi(\ff)\subset \p^3(\fqc)$, where $\phi=(1:x:y:z)$. Since

\begin{equation}
\left|
\begin{array}{cccc}
  1 & x & y & z  \\
  0 & 1 & D_x^{(1)}(y)& D_x^{(1)}(z) \\
	0 & 0 & 0& D_x^{(2)}(z) \\
	0 & 0 & D_x^{(i)}(y)& D_x^{(i)}(z) 
  \end{array}
  \right|=-D_x^{(2)}(z)\cdot D_x^{(i)}(y),
\end{equation}
for $i>1$, we conclude that the order sequence of $\xx$ is $(0,1,2,p^r)$. Futhermore, we have $\nu_2=\varepsilon_3$, since for $j>1$ the equality

\begin{equation}
\left|
\begin{array}{cccc}
  1 & x^q & y^q & z^q  \\
	1 & x & y & z \\
  0 & 1 & D_x^{(1)}(y)& D_x^{(1)}(z) \\
	0 & 0 & D_x^{(j)}(y)& D_x^{(j)}(z) 
  \end{array}
  \right|=D_x^{(j)}(y)\cdot \big(D_x^{(j)}(z)(x-x^q)-(z-z^q)\big)
\end{equation}
gives that the $\fq$-Frobenius order sequence of $\xx$ is $(0,1,p^r)$. Therefore, we conclude from Proposition \ref{oschypeq} that the Gauss map of $\xx$ is given by $\gamma=(z_2^{p^r}: z_0^{p^r}:z_1^{p^r}:0 )$. Since $\ff$ is $\fq$-Frobenius nonclassical, we have that the morphism $(z_0:z_1:z_2)$ is birational, and so $\gamma(\xx)=\xx^{'}$ is birationally equivalent to $\ff$ up to the $\Phi_{p^r}$-Frobenius map. In particular, $\deg_s \gamma=k$. Finally, if $\lambda \in \fqc$ denotes a $k$-primitive root of unity, we clearly have that the points $(1:a:b:\lambda^j c) \in \xx$ with $j=0,\ldots,k-1$ have the same osculating hyperplane and are collinear. 

\begin{ex}
Let $q$ be an even power of $p>2$ and set $\ff:x^{\sqrt{q}+1}+y^{\sqrt{q}+1}+1=0$ defined over $\F_q$ (Hermitian curve over $\F_q$). It is well known that $\ff$ is $\fq$-Frobenius nonclassical, see \cite[Theorem 2]{GV2} for instance. For $k>1$ prime to $p$, the function $f=\frac{x^2}{y}$ is such that $f \neq w^d$ for all $w \in \fqc(x,y)$ and $d$ divisor of $k$, as there exists a place $Q$ of $\fqc(x,y)$ such that $v_Q(f)=-1$ (\cite[Corollary 3.7.4]{St}). Let $z$ be such that $z^k=\frac{x^2}{y}$. One can show that $D_x^{(2)}(z) \neq 0$. Hence, by the discussion above, $\phi(\ff)\subset \p^3(\fqc)$, where $\phi=(1:x:y:z)$ is such that $\varepsilon_3=\nu_2=\sqrt{q}$, $\deg_s \gamma=k$ and the generic points of $\xx$ sharing the same osculating hyperplane are collinear.
\end{ex}

One clearly has that the result of Theorem \ref{gen1} does not hold without the assumption $\nu_{n-1}=\varepsilon_n$. For instance, as discussed previously in this paper, it is known that if a plane curve $\xx$ is nonsingular and $\deg_s \gamma>1$, then $\xx$ is Frobenius classical. For sake of completeness, we will present the following example for curves in higher dimensional spaces. It is based on \cite[Example 5]{GV0}.

\begin{ex}
Consider a  curve $\ff$ defined by $x^6=f(y)$ over $\F_p$, where $f(T) \in \F_p[T]$, with $p > 3$ and set $\xx=\psi(\ff) \subset \p^3(\overline{\F}_p)$, where $\psi=(1:x:x^2:x^{3p}+y^px+x^{6p+2})$. Then $\gamma=(x^{3p}:y^p:x^{6p}:-1)$, and so $\deg_s \gamma=3$. Moreover, a generic pre-image of $\gamma$ is not collinear. Indeed, let $P_1=(1:a:a^2:a^{3p}+b^pa+a^{6p+2}) \in \xx$. Then $P_2=(1:\xi a:(\xi a)^2:a^{3p}+b^p(\xi a)+(\xi a)^{6p+2})$ and $P_3=(1:\xi^2 a:(\xi^2 a)^2:a^{3p}+b^p(\xi^2 a)+(\xi^2 a)^{6p+2})$ have the same image $(a^{3p}:b^p:a^{6p}:-1)$ under $\gamma$, where $\xi$ is a primitive cubic root of $1$. However, one can easily check that such points are not collinear.
\end{ex}

We can also specialize Theorem \ref{main} in the case $\ell=n-1$.

\begin{thm}\label{n-1}
Let $\xx \subset \p^n(\fqc)$ be a curve defined over $\fq$ with order sequence $(\varepsilon_0,\ldots,\varepsilon_n)$. Suppose that $\xx$ is $\F_{q^i}$-Frobenius nonclassical for $i=1,\ldots,n-1$ such that the last term of the $\F_{q^i}$-Forbenius order sequence equals to $\varepsilon_n$. Then $\deg_s \gamma=1$, that is, the Gauss map $\gamma$ is purely inseparable.
\end{thm}
\begin{proof}
The case $n=2$ is discussed in Remark \ref{obsgen}. Assume that $n>2$ and $\deg_s \gamma >1$. Let $P \in \xx$ be a generic point and let $Q \in \xx$ such that $P \neq Q$ and $H_P(\xx)=H_Q(\xx)$. Denote by $d$ the dimension of the linear subspace $L \subset \p^n(\fqc)$ defined by the set $\{P,\Phi_q(P),\ldots,\Phi_{q^{n-2}}(P),Q,\Phi_q(Q),\ldots,\Phi_{q^{n-2}}(Q)\}$. By Theorem \ref{main}, one has $d \leq n-2$. Suppose that $d<n-2$. Then the set $A=\{P,\Phi_q(P),\ldots,\Phi_{q^{n-2}}(P)\}$ determines a linear subspace $\tilde{L}$ of dimension $k \leq d < n-2$ and $\Phi_q(R) \in A \subset \tilde{L}$ for all $R=\phi_{q^j}(P)$, $j=0,\ldots,n-3$. Note that since $A$ has $n-1$ points, we have a redundancy of points to generate $\tilde{L}$. If $\tilde{L}$ is defined by the set $A \backslash\{\Phi_{q^{n-2}}(P)\}$, then  $\tilde{L}$ is fixed by $\Phi_q$, which gives that there are infinitely many linear subspaces of degree $k$ defined over $\fq$, a contradiction. If $\tilde{L}$ is not defined by $A \backslash\{\Phi_{q^{n-2}}(P)\}$, then the subspace determined by such set of points has dimension smaller than $k$, and by induction we arrive again at a contradiction.

Suppose then $d=n-2$. In this case, we may assume that $L$ is defined by the set $A$, because otherwise $A$ would determine a linear subspace of dimension smaller then $n-2$ and we would obtain a contradiction proceeding as above. If $B=\{Q,\Phi_q(Q),\ldots,\Phi_{q^{n-3}}(Q)\}$ is not contained in the linear subspace determined by $C=\{P,\Phi_q(P),\ldots,\Phi_{q^{n-3}}(P)\}$, then there exists  $j \in \{0,\ldots,n-3\}$ such that $L$ is defined by $C \cup \{\Phi_{q^j}(Q)\}$. Therefore, $L$ is fixed by $\Phi_q$, which is a contradiction. If $B$ is contained in the the linear subspace determined by $C$, an inductive argument as used previously also gives rise to a contradiction. Hence $\deg_s \gamma =1$, as claimed.
\end{proof}

\begin{ex}
Suppose $p>3$ and consider the irreducible plane curve $\ff$ defined by the affine equation
$$
y^q-y=\frac{1-x^{q-1}}{1+x^{q-1}}
$$ over $\fq$. Then $x$ is a separating element of $\fq(\ff)=\fq(x,y)$. Let $z$ be a root of the polynomial
$$
T-T^q+(x-x^q)x+(y-y^q)x^2 \in \fq(x,y)[T]
$$ 
and denote by $\cc$ the curve defined over $\fq$ such that $\fq(\cc)=\fq(x,y,z)$. Note that $x$ is separating in $\fq(\cc)$. Thus the following equations are satisfied:
\[\begin{cases}
(z-z^q)+(x-x^q)x+(y-y^q)x^2   & =0, \\
(z-z^{q^2})+(x-x^{q^2})x^q+(y-y^{q^2})x^{2q}   & =0. 
\end{cases}\]
Finally, let $\psi=(1:x:x^2:z^q+x^{q+1}+y^qx^2):\cc \lra \p^3(\fqc)$ and define $\xx=\psi(\cc)$. The order sequence of $\xx$ is $(0,1,2,q)$. By the equations above we conclude that the $\fq$ and $\F_{q^2}$-Frobenius order sequences are the same, namely $(0,1,q)$. Therefore $\xx$ satisfies the hypotheses of Theorem \ref{n-1}.
\end{ex}

We end this section by pointing out that if we do not require that $i=1,\ldots,n-1$ in Theorem \ref{n-1}, then the result does not necessarily holds. To this end, we present two examples of curves in $\p^3(\fqc)$ (based on the construction given in page 6) satisfying the hypothesis of Theorem \ref{n-1} except that this time we do not have $i=1,2$, such that one has purely inseparable Gauss map and the other does not. 

\begin{ex}
Suppose $p>3$. Let $u,m$ be co-prime positive integers with $m >2$ and $m>u$, and consider the Borges curve defined over $\fq$ by $\ff:f(x,y)=0$, where
$$
f(x,y)=\frac{(x^{q^u}-x)(y^{q^m}-y)-(x^{q^m}-x)(y^{q^u}-y)}{(x^{q^2}-x)(y^{q}-y)-(x^{q}-x)(y^{q^2}-y)}.
$$
This curve was introduced in \cite{Bo2}, and it was proved that it is the only simultaneously $\F_{q^m}$- and  $\F_{q^u}$-Frobenius nonclassical curve for the morphism of lines (\cite[Theorem 1.1]{Bo2}), and its order sequence is $(0,1,q^u)$. 


 Let $\xx=\psi(\ff)$, where $\psi=(1:x:y:xy)$. Since $D_x^{(2)}(xy)=y+xD_x^{(1)}(y) \neq 0$, we have that 
 the order sequence of $\xx$ is $(0,1,2,q^u)$ and the $\F_{q^r}$-Frobenius order sequence is $(0,1,q^u)$ for $r=u,m$. 
  Thus $\gamma$ is purely inseparable, as $\fqc(\xx)=\fqc(\ff)$ (recall that in such situation we have $\fqc(\gamma_s(\xx))=\fqc(\gamma_s(\ff))=\fqc(\ff)$).

Now let $P=(a:b:1) \in \ff$ be a generic point such that $x-a$ is a local parameter at $P$ (is not so difficult to check that there exists such point). Let $k>1$ be an integer such that $p \nmid k(k+1)$. Set $w=\frac{1}{x-a}$, and let $z$ be a root of $T^k-w$. Define $\varphi=(1:x:y:z): \ff \lra \p^3(\fqc)$ and set $\yy=\varphi(\ff)$. Thus $[\fqc(\yy):\fqc(\ff)]=k$. The condition $p \nmid k+1$ ensures that $D_x^{(2)}(z) \neq 0$. Hence the order sequence  of $\yy$ is $(0,1,2,q^u)$, and $\yy$ is $\F_{q^r}$-Frobenius nonclassical, for $r=u,m$. However, one has $\deg_s \gamma=k>1$. 
\end{ex}

\section{Strict dual curves covered by a quotient curve}

Let $\xx \subset \p^n(\fqc)$ be a curve defined over $\fq$ equipped with the Frobenius map $\Phi_q$. In Section 3 we took advantage of the fact that $\Phi_q$ induces a Frobenius map on $\xx^{'}$ such that $\gamma \circ \Phi_q=\Phi_q \circ \gamma$, to bound $\deg_s \gamma$. In this section, we use a similar approach for linear automorphisms. 

Let $(\varepsilon_0,\ldots,\varepsilon_n)$ the order sequence of $\xx$ and assume that $\varepsilon_n=p^r$, where $n>p$. Recall that in such case, $\gamma=\Phi_{p^r}\circ \gamma_s$, and the extension $\fqc(\xx) / \fqc(\gamma_s(\xx))$ is separable. In this section, we will investigate what happens if there exists an intermediate field $\fqc(\gamma_s(\xx)) \subseteq \mathbb{F} \subsetneq \fqc(\xx)$ such that the extension $\fqc(\xx)/\F$ is  Galois defined by a linear automorphism group $G$, specially when $\xx$ is a plane curve. In what follows, $\aut(\xx)$ will denote the full automorphism group of $\xx$.

\begin{prop}\label{degen}
Let $\xx \subset \p^n(\fqc)$ defined over $\fq$ such that $\varepsilon_n=p^r$, where $n>p$. If $\gamma_s(\xx)$ is a subcover of a quotient curve $\xx/G$, for some linear nontrivial subgroup $G<\aut(\xx)$, then $\xx^{'}$ is degenerate. In particular, if $\xx$ is a plane curve, then $\xx$ is strange.
\end{prop}
\begin{proof}
Since $G$ is linear, $G$ induces a linear automorphism group in $\xx^{'}$, as for a nontrivial $\sigma \in G$ and a generic point $P \in \xx$, one has $H_{\sigma(P)}(\xx)=\sigma(H_P(\xx))$. Since $\xx^{'}$ differs from $\gamma_s(\xx)$ by a Frobenius colineation, one has an induced  linear action by $G$ in $\gamma_s(\xx)$. Let $P \in \xx$ be a generic point. Then the orbit of $P$ by $G$, namely $G(P)$, lies over the same point $Q \in \xx/G$. Note that since  $G$ is nontrivial, we have $\#G(P)>1$. In its turn, $Q$ lies over a point $\tilde{P} \in \gamma_s(\xx)$. But $\tilde{P}$ corresponds to the osculating hyperplane of all point in $G(P)$, and then $Q$ is fixed by the action induced by $G$. This implies that the group $G$ acting on $\gamma_s(\xx)$ has infinite fixed points. Hence, $G$ is trivial in $\aut(\gamma_s(\xx))$. Therefore, $\gamma_s(\xx)$ is contained in the linear subvariety fixed by $G$ in the dual projective space. 
\end{proof}

\begin{rem}
Recall that every automorphism in a nonsingular curve is linear.
\end{rem}

We now will exploit the above situation for plane curves. If $p>2$ and $\xx$ is a plane strange non-rational curve, then $\deg_s\gamma>1$, and so $\xx$ is nonclassical. Hence, there are $z_0,z_1,z_2 \in \fqc(\xx)$ not all zero and $r>0$ such that
\begin{equation}\label{forma}
z_0^{p^r}x+z_1^{p^r}y+z_2^{p^r}=0.
\end{equation}
Here, $\deg_i \gamma =p^r$ and $\gamma_s=(z_0:z_1:z_2)$. Up to a projective transformation, we may assume that a generic tangent to $\xx$ passes trough $(1:0:0)$. In other words, we may assume $z_0=0$, and $\xx$ is defined by a polynomial dividing
\begin{equation}\label{eqgal}
(Z_1(X,Y,Z))^{p^r}Y+(Z_2(X,Y,Z))^{p^r}Z
\end{equation}
where $Z_i(X,Y,Z) \in \fqc[x,y,z]$ are homogeneous, $i=1,2$, and $z_1, z_2$ and $y$ are the projections of $Z_1,Z_2$ and $Y$ in $\fqc(\xx)$, respectively. 

\begin{thm}\label{gaussgal}
Let $\xx$ be plane curve defined over $\fq$, with $p>2$, such that $\gamma_s(\xx)$ is a subcover of a quotient curve $\xx/G$, for some linear nontrivial subgroup $G<\aut(\xx)$. Then, up to a projective transformation over $\fqc$, $\xx$ is defined by a polynomial dividing \eqref{eqgal} and the following hold:
\begin{itemize}
\item[(i)]  For all $\sigma \in G$, there are $a,b,c \in \fqc$ (depending on $\sigma$) such that 
\begin{equation}\label{shape}
\sigma:(X:Y:Z)\mapsto (aX+bY+cZ:Y:Z).
\end{equation}
\item[(ii)] $G=E_p \rtimes C$ where $E_p$ is an elementary abelian $p$-group and $C$ is a cyclic group  such that $|C|$ divides  $|E_p|-1$. Thus in \eqref{shape} we have $a^{|E_p|-1}=1$ and $a=1$ if and only if $\sigma$ has order $p$.   
 \end{itemize}
\end{thm}
\begin{proof}
Since $G$ is nontrivial, from Proposition \ref{degen}, we know that $\xx$ is an strange curve. Thus $\xx$ is defined by an irreducible factor of \eqref{eqgal}. Since $\fqc(\gamma_s(\xx))=\fqc(z_2/z_1) \subset \fqc(\xx)^G$, then $\sigma \in G$ fixes $z_2/z_1$. We conclude that $\sigma(y)=y$ for all $\sigma \in G$. Moreover, since $G$ preserves a generic tangent of $\xx$ and all such tangents meet at $(1:0:0)$, we obtain that $G$ fixes $(1:0:0)$. Therefore, we obtain \eqref{shape}

Finally, considering a generic tangent line $\ell$ to $\xx$ (not fixed pointwise by $G$), we have that the restriction of $G$ to $\ell$ is an automorphism of $\ell$ fixing the point $P=(1:0:0)$. Hence $G$ is isomorphic to a subgroup of the stabilizer of a point on the projective line. Thus it follows from \cite[Theorem 1]{VM} that $G=E_p \rtimes C$ where $E_p$ is an elementary abelian $p$-group and $C$ is cyclic with order dividing $|E_p|-1$. The last statement is straightforward.  
\end{proof}

\begin{ex}\label{exest}
If $p=5$, the curve $\ff$ given by the affine equation $(yx^3+2y)^5y+(x^3+2y)^5=0$ is irreducible and has order sequence $(0,1,5)$. Such curve satisfies the hypothesis of Theorem \ref{gaussgal}, with the automorphism group generated by $(x,y) \mapsto (\xi x, y)$, where $\xi$ is a primitive cubic root of unity.
\end{ex}

The next result shows that the equation defining the curve $\ff$ is example \ref{exest} is not so specific.

\begin{thm}\label{gaussgal2}
Let $\xx$ be plane curve defined over $\fq$, with $p>2$, such that $\gamma_s(\xx)$ is a subcover of a quotient curve $\xx/G$, for some linear nontrivial subgroup $G<\aut(\xx)$ of order $n$. If $p \nmid |G|$, then up to a projective transformation over $\fqc$, $\xx$ is defined by an irreducible factor of a polymomial of the type
\begin{equation}\label{geral}
z_1\big(x^n,y\big)^{p^r}y+z_2\big(x^n,y\big)^{p^r},
\end{equation} 
with $z_1,z_2 \in \fqc[x,y]$, $r>0$ and $G=\langle \sigma \rangle$, with $\sigma:(x,y) \mapsto (\xi x, y)$, for a primitive $n$-th root of unity $\xi \in \fqc$. 
\end{thm}
\begin{proof}
Suppose that $|G|$ is not divisible by $p$. By Theorem \ref{gaussgal}, if $\xx:f(x,y)=0$, we have that $\fqc(\xx)=\fqc(x,y)$, there are $z_1,z_2 \in \fqc(\xx)$ with $z_2/z_1 \in \fqc(x,y)^G$ such that $z_1^{p^r}y+z_2^{p^r}=0$ and $G$ is a cyclic group generated by an automorphism $\sigma:(x,y)\mapsto (ax+by+c,y)$ with $a \neq 1$. Consider the change of coordinates defined by 
$$
T:(x,y) \mapsto (\bar{x},y)=\left(x+\frac{b}{a-1}y+\frac{c}{a-1},y\right).
$$
Then $\tilde{\sigma}:=T \circ \sigma \circ T^{-1}:(\bar{x},y)=(a \bar{x},y)$. Moreover, $z_i(x,y)=z_i(\bar{x}-\frac{b}{a-1}y-\frac{c}{a-1},y)=:\tilde{z_i}(\bar{x},y) \in \fqc(\bar{x},y)$, and one can check directly that $\tilde{\sigma}(\tilde{z_2}(\bar{x},y)/\tilde{z_1}(\bar{x},y))=\tilde{z_2}(\bar{x},y)/\tilde{z_1}(\bar{x},y)$. Now, we clearly have $\fqc(\bar{x}^n,y) \subset \fqc(\bar{x},y)^G$, and since $|G|=n$ equality holds. Since $\tilde{z_2}(\bar{x},y)/\tilde{z_1}(\bar{x},y) \in \fqc(\bar{x},y)^G$, we conclude that $\tilde{z_2}(\bar{x},y)/\tilde{z_1}(\bar{x},y)=z^{'}(\bar{x}^n,y)$, for some rational function $z^{'}$. Finaly,
$$
y+z^{'}(\bar{x}^n,y)^{p^r}=y+\frac{z_2(x,y)^{p^r}}{z_1(x,y)^{p^r}}=0.
$$
Therefore, the result follows from a formal replacement of $\bar{x}$ with $x$ and $z^{'}$ with $z_2/z_1$.
\end{proof}

We now give a characterization of equation \eqref{eqgal} when $p$ divides $|G|$. Let $\xx$ and $G$ as in Theorem \ref{gaussgal}, and let $E_p<G$ be the $p$-Sylow subgroup of $G$ of order $p^s$. Let $m$ be smallest integer for which $G$ is defined over $\F_{p^m} \supseteq \fq$. In particular, $s \leq m$, see \cite[Theorem 3]{VM}. By Theorem \ref{gaussgal}, every element of $E_p$ is defined by $\tau_{b,c}:(x,y) \mapsto (x+by+c,y)$, with $b,c \in  \F_{p^m}$. Hence $E_p$ is isomorphic to the additive group
$$
\Lambda:=\{(b,c) \ | \ \tau_{b,c} \in G \}< \F_{p^m} \times \F_{p^m}.
$$

\begin{thm}
Let $\xx$ be plane curve defined over $\fq$, with $p>2$, such that $\gamma_s(\xx)$ is a subcover of a quotient curve $\xx/G$, for some linear nontrivial subgroup $G<\aut(\xx)$. Suppose that $p$ divides $|G|$ and $\F_{p^m}$ is the smallest extension of $\fq$ for which $G$ is defined. Then there exists $n>0$ such that, up to a projective transformation over $\fqc$, $\xx$ is defined by an irreducible factor of a polynomial of the type
\begin{equation}\label{geral2}
z_1\Big(\prod_{(b,c) \in \Lambda}(x+by+c)^n,y\Big)^{p^r}y+z_2\Big(\prod_{(b,c) \in \Lambda}(x+by+c)^n,y\Big)^{p^r},
\end{equation} 
with $z_1,z_2 \in \fqc[x,y]$, $r>0$ and $\Lambda <\F_{p^m} \times \F_{p^m}$ is an additive subgroup of $\F_{p^m} \times \F_{p^m}$ of order $p^s \leq p^m$ such that $n$ divides $\gcd(p^s-1,p^m-1)$. Moreover, $G=E_p \rtimes \langle \sigma \rangle$, with $\sigma:(x,y) \mapsto (\xi x, y)$ for a primitive $n$-th root of unity $\xi \in \fqc$ and $E_p=\{\tau_{b,c}:(x,y) \mapsto (x+by+c,y) \ | \ (b,c) \in \Lambda\} \cong \Lambda$ is elementary abelian $p$-group. 
\end{thm}
\begin{proof}
By Theorem \ref{gaussgal} and the proof of Theorem \ref{gaussgal2}, we may assume that \eqref{eqgal} holds and $G= E_p \rtimes \langle \sigma \rangle$, with $\sigma:(x,y) \mapsto (\xi x, y)$ for a primitive $n$-th root of unity $\xi \in \fqc$, $E_p$ is elementary abelian $p$-group of order $p^s$ and $n | p^s-1$. In fact, using notation of Theorem \ref{gaussgal2}, one can check that $T \circ \tau_{d,e} \circ T^{-1}=\tau_{d,e}$ and $\tau_{d,e}(\tilde{z_i}(\bar{x},y))=\tilde{z_i}(\bar{x},y)$ for $i=1,2$. Here, if $m$ is the highest power of $p$ for which $G$ is defined over $\F_{p^m}$, we have $q \leq p^m$, and by \cite[Theorem 3]{VM}, we have $s \leq m$ and $n| p^m-1$. In particular, multiplication  by $\xi$ gives a faithful action on $\F_{p^m}$. Now, the function 
$$
w:=\prod_{(b,c) \in \Lambda}(x+by+c)^n
$$ 
is fixed by every element of $G$. Indeed,
$$
\sigma(w)=\prod_{(b,c) \in \Lambda}(\xi x+by+c)^n=\prod_{(b/\xi,c/\xi) \in \Lambda}\xi^n( x+(b/\xi)y+c/\xi)^n=w,
$$
and
$$
\tau_{d,e}(w)=\prod_{(b,c) \in \Lambda}(x+(b+d)y+(c+e))^n=\prod_{(b+d,c+e) \in \Lambda}(x+(b+d)y+(c+e))^n=w.
$$
Thus $\fqc(w,y) \subseteq \fqc(\xx)^G$. On the other hand, $[\fqc(\xx):\fqc(w,y)] \leq p^sn=|G|$, and so $\fqc(w,y)=\fqc(\xx)^G$. Since $z_2/ z_1 \in \fqc(\xx)^G$, we obtain the claimed result.
\end{proof}

\begin{rem}
Let $\xx$ be a plane curve of degree $d$ defined over $\fqc$. We say that $\xx$ has controlled singularities if 
$$
\sum_{P \in \xx}e_P<\frac{1}{2}d,
$$ where $e_P$ denotes the multiplicity of the Jacobian ideal of $\xx$ at $P$. In \cite[Corollaries 5.10 and 5.16]{He}, it is shown that if $\xx$ is nonlcassical with order sequence $(0,1,p^r)$ and it has controlled singularities, then $\xx$ has an equation of the form \eqref{forma}. However, one can check that strange curves does not have controlled singularities. Thus such criteria cannot be applied to our case.
\end{rem}

\begin{rem}
Given a plane curve $\xx$, a point $P \in \p^2(\fqc)$ is called a Galois point of $\xx$ if the projection $\pi_P: \xx \lra \p^1(\fqc)$ from $P$ induces a Galois extension of function fields $\fqc(\xx)/\fqc(\p^1)$, see for instance \cite{Fu,Yo}.  One can check that if $\fqc(\gamma_s(\xx))=\fqc(\xx)^G$ for a linear automorphism group $G$, then the nucleus of $\xx$ (which exists, since in such case $\xx$ is strange) is a Galois point of $\xx$.
\end{rem}

\section*{Acknowledgements}
This research was partially supported by FAPESP-Brazil, grant 2016/24713-4.


\printindex

\end{document}